\documentclass[11pt,reqno]{amsart}

\usepackage{amsxtra}
\usepackage{amssymb}

\newcommand\fg{{\mathfrak g}}

\newcommand\CC{\mathbb C}

\newcommand\RR{\mathbb R}

\theoremstyle{plain}

\newtheorem{thm}{Theorem}[section]
\newtheorem{lem}[thm]{Lemma}
\newtheorem{prop}[thm]{Proposition}
\newtheorem{cor}[thm]{Corollary}

\theoremstyle{definition}
\newtheorem{defn}[thm]{Definition}

\theoremstyle{remark}
\newtheorem{rem}{Remark}

\newtheorem{ex}{Example}

\title{complex connections with trivial holonomy}

\author{A. Andrada}
\email{andrada@mate.uncor.edu}
\author{M.L. Barberis}
\email{barberis@mate.uncor.edu}
\author{I.G. Dotti}
\email{idotti@mate.uncor.edu}

\address{FaMAF-CIEM, Universidad Nacional de C\'{o}rdoba, Ciudad Universitaria, 5000 C\'{o}rdoba, Argentina}
\thanks{The authors were partially supported by
CONICET, ANPCyT and SECyT-UNC (Argentina).}

\begin{document}
\maketitle

\begin{abstract}
Given  an almost complex manifold $(M, J)$, we study complex
connections   with trivial holonomy and such that the
corresponding torsion  is either of type $(2,0)$ or of type
$(1,1)$ with respect to $J$. Such connections arise naturally when
considering  Lie groups, and quotients by discrete subgroups,
equipped with bi-invariant and abelian complex structures.
\end{abstract}

\section{Introduction}

Let $M$ be an $n$-dimensional  connected manifold together with an
affine connection   with trivial holonomy, hence flat. This
amounts to having  an absolute parallelism on $M$, which in turn
is equivalent to a smooth trivialization of the frame bundle $B$
(see  \cite[Proposition 2.2]{W}). Given a pseudo-Riemannian metric
$g$ on $M$, J. Wolf studied in \cite{W,W2} the problem of
existence of  metric connections with trivial holonomy and having
the same geodesics as the Levi-Civita connection. Equivalently, he
considered the class of pseudo-Riemannian manifolds $(M,g)$ which
carry connections $\nabla$ such that $\nabla g=0$, Hol($\nabla$)=
$1$ and whose torsion is totally skew-symmetric (see \cite{AF} for
a different approach  using geometries with torsion in the
Riemannian case). When the connection is required to be complete
with parallel torsion, the resulting manifolds are of the form
$\Gamma\backslash G$ with $G$ a simply connected Lie group
carrying a bi-invariant pseudo-Riemannian metric and $\Gamma$ a
discrete subgroup of $G$. Moreover, $\nabla$ is induced by the
affine connection corresponding to the parallelism of left
translation on $G$ and the pseudo-Riemannian metric $g$ is induced
from a bi-invariant metric on $G$ \cite[Theorem 3.8]{W}. He also
provided a complete classification of all complete
pseudo-Riemannian manifolds admitting such connections in the
reductive case \cite[Theorem 8.16]{W2}.

It is our aim to investigate an analogue of the previous problem
in the case of almost complex manifolds instead of
pseudo-Riemannian manifolds. More precisely, given  an almost
complex manifold $(M, J)$, we will be interested in studying
complex connections $\nabla$ on $M$ with trivial holonomy and such
that the corresponding torsion $T$ is either of type $(2,0)$ or of
type $(1,1)$ with respect to $J$.

Our first observation (see Proposition \ref{cor(1,1)}) is that
when $\nabla J =0$ and  the torsion $T$ of $\nabla$ is of type
$(2,0)$ or  $(1,1)$ then $J$ is necessarily integrable, that is,
$(M,J)$ is a complex manifold. We prove a general result for
affine complex manifolds with trivial holonomy (Theorem
\ref{discrete}),
 analogous to \cite[Theorem 1]{Wa} (see also
\cite[Theorem 3.6]{KT}).

We show that on a complex parallelizable manifold the existence of
$n$ holomorphic vector fields $Z_1, \dots , Z_n$ which are
linearly independent at every point of $M$ is equivalent to the
existence of a complex connection $\nabla$ on $M$ with trivial
holonomy whose torsion tensor field $T$ is of type $(2,0)$
(Proposition \ref{complex-parallelizable}).  In particular the
Chern connection of any metric compatible with $J$ and having
constant coefficients on the given trivialization has trivial
holonomy, hence it is flat.

On the other hand, the existence of  $n$ commuting vector fields
$Z_1, \dots , Z_n$ which are linearly independent sections of
$T^{1,0}M$  at every point of $M$ is equivalent to the existence
of a complex connection $\nabla$ on $M$ with trivial holonomy
whose torsion tensor field $T$ is of type $(1,1)$ (Proposition
\ref{def-ab}).
 Such a connection will be called an {\em abelian connection}.
The class of abelian connections is the $(1,1)$-counterpart of the
class of Chern-type connections on complex parallelizable
manifolds (Proposition \ref{complex-parallelizable} and Definition
\ref{def-chern-type}).

Our motivation for studying  abelian connections arises from the
fact that quotients by a discrete subgroup of  Lie groups carrying
abelian complex structures, which have been studied by several
authors (see, for instance \cite{ABD,BD2,cfp,cfu,df_0,mpps,V}),
are natural examples of such manifolds. One of our results
(Corollary \ref{2-step-solv}), which is analogous to \cite[Theorem
3.8]{W}, asserts that when $\nabla$ is complete with parallel
torsion then $M=\Gamma\backslash G$ with $G$ a simply connected
Lie group, $\Gamma$ a discrete subgroup of $G$, $\nabla$ is
induced by the connection corresponding to the parallelism of left
translation on $G$ and the almost complex structure $J$ comes from
a left invariant  abelian complex structure on $G$.

We point out that the class of complex connections we consider in
this work is not the same as the class of complex-flat connections
introduced by D. Joyce in \cite{Joy}. According to \cite{Joy}, an
affine connection $\nabla$ on a  complex manifold $(M,J)$ is
called complex-flat when $\nabla J=0$, $\nabla$ is torsion-free
and the curvature tensor of $\nabla$ satisfies certain condition
which is always fulfilled by the curvature tensor of a K\"ahler
metric. The tangent bundle $TM$ of $M$ is naturally a complex
manifold, with a complex structure $\mathcal J$ induced by $J$. It
was shown in \cite[Theorem 6.2]{Joy} that given a complex-flat
connection on $(M,J)$ it is possible to endow  $TM$ with a complex
structure $\mathcal K$ commuting with $\mathcal J$. In this case,
both $\mathcal J$ and $\mathcal K$ induce complex structures on
the cotangent bundle $T^*M$ and it turns out that the natural
symplectic structure on  $T^*M$ gives rise to a pseudo-Riemannian
metric on $T^*M$ which is pseudo-K\"ahler with respect to
$\mathcal K$.

\

\noindent
 {\sc Acknowledgement.} We dedicate this article to Joe
Wolf, whose work has inspired the research of many mathematicians
in the broad field of Differential Geometry and Lie Theory. We are
grateful for his important contribution to the development of our
Mathematics Department.

\

\section{Preliminaries}

Let $\nabla$ be an affine connection on a  manifold $M$ with
torsion tensor field $T$, where $T(X,Y)= \nabla_XY - \nabla_YX
-[X,Y]$, for all $X, Y$ vector fields on $M$.

Given an almost complex structure $J$ on $M$, we denote by $N$ the
Nijenhuis tensor of $J$, defined by
\begin{equation}\label{nijen}
N(X,Y):=[JX,JY]-J[X,JY]-J[JX,Y]- [X,Y].
\end{equation}
Recalling that
\begin{equation}\label{parallel}
\left( \nabla_{X}J\right)Y = \nabla _X \left( JY\right) - J
\left(\nabla _X Y \right),
\end{equation}
we obtain the following identity:
\begin{equation}\label{nijen1}
\begin{split}
N(X,Y) &\;= \left( \nabla_{JX}J\right)Y -\left(
\nabla_{JY}J\right)X + \left( \nabla_{X}J\right)JY -\left(
\nabla_{Y}J\right)JX
\\ & \qquad +T(X,Y)-T(JX,JY)+ J\left( T(JX,Y)+T(X,JY) \right),
\end{split}
\end{equation}
for all $X,Y$ vector fields on $M$.

The almost complex structure $J$ is called integrable when
$N\equiv 0$, and in this case $(M,J)$ is a complex manifold
\cite{NN}. The tensor field $J$ is called parallel with respect to
$\nabla$ when $\nabla J = 0$, that is, $ \left(\nabla_{X}J\right)Y
=0$ for all $X,Y$ vector fields on $M$ (see \eqref{parallel}).
Also in this case one says that $\nabla$ is a complex connection
(see \cite[p. 143]{KN}).

The next lemma follows from equation \eqref{nijen1}.
\begin{lem}\label{lem(1,1)}
Let $(M,J)$ be an almost complex manifold with a complex
connection $\nabla $. Then $J$ is integrable if and only if the
torsion $T$ of $\nabla$ satisfies:
\[  T(X,Y)-T(JX,JY)+ J\left(T(JX,Y)+T(X,JY)\right) =0,\]
for all vector fields $X,Y$ on $M$.
\end{lem}

\smallskip

The torsion $T$ of a connection $\nabla$ on the almost complex
manifold $(M,J)$ is said to be:
\begin{itemize}
\item of type $(1,1)$ if $T(JX,JY)=T(X,Y)$,
\item of type $(2,0)$ if $T(JX,Y)=JT(X,Y)$,
\item of type $(2,0)+(0,2)$ if $T(JX,JY)=-T(X,Y)$,
\end{itemize}
for all vector fields $X,Y$ on $M$.

\smallskip

\begin{prop}\label{cor(1,1)}
Let $(M,J)$ be an almost complex manifold.
\begin{enumerate}
\item[(i)] If $\nabla$ is a complex connection on $M$ whose torsion is of
type $(1,1)$ with respect to $J$, then $J$ is integrable.
\item[(ii)] If $\nabla$ is a complex connection on $M$ whose torsion is of
type $(2,0)$ with respect to $J$, then $J$ is integrable.
\item[(iii)] If $\nabla$ is a complex connection on $M$ whose torsion is of
type $(2,0)+(0,2)$ and $J$ is integrable, then $T$ is of type
$(2,0)$.
\item[(iv)] If $J$ is integrable, then there exists a complex
connection $\nabla$ whose torsion is of type $(1,1)$ with respect
to $J$.
\item[(v)] If $J$ is integrable, then there exists a complex connection
$\nabla$ whose torsion is of type $(2,0)$ with respect to $J$.
\end{enumerate}
\end{prop}

\begin{proof}
(i), (ii) and (iii) are a straightforward consequence of Lemma
\ref{lem(1,1)}.

To prove (iv) and (v), we introduce a Hermitian metric $g$ on $M$,
that is, $g$ is a Riemannian metric on $M$ satisfying $g(JX,JY)=
g(X,Y)$ for all vector fields $X,Y$ on $M$. If $\nabla^g$ is the
Levi-Civita connection of $g$, then we consider the connections
$\nabla^1$ and $\nabla^2$ defined by
\begin{equation}\label{first} g\left( \nabla^1_X Y , Z \right) =  g\left(
\nabla^g _X Y , Z \right) + \frac 14 \left( d\omega (X,JY,Z) +
d\omega (X,Y,JZ) \right) , \end{equation} 
\[ g\left( \nabla^2_X Y , Z \right) =  g\left( \nabla^g _X Y ,Z
\right) - \frac12 \, d\omega (JX,Y,Z), \] where $\omega (X,Y):=
g(JX,Y)$ is the K\"ahler form corresponding to $g$ and $J$. These
connections satisfy
\[ \nabla^1 g =0, \quad \nabla^1 J= 0, \quad T^1  \; \text{ is of type }
(1,1), \] 
\begin{equation}\label{chern} \nabla^2 g =0, \quad
\nabla^2 J= 0, \quad T^2 \; \text{ is of type } (2,0),
\end{equation}
(see \cite{Li,Ga}), thus proving the claim.
\end{proof}

\begin{rem}
The connections $\nabla^1$ and $\nabla^2$ appearing in the proof
of Proposition \ref{cor(1,1)} are known, respectively, as the
first and second canonical connection associated to the Hermitian
manifold $(M,J,g)$. The connection $\nabla^2$ is also known as the
{\em Chern} connection, and it is the unique connection on
$(M,J,g)$ satisfying \eqref{chern}. In the almost Hermitian case,
the Chern connection is the unique complex metric connection whose
torsion is of type $(2,0)+(0,2)$, equivalently, the
$(1,1)$-component of the torsion vanishes.
\end{rem}

\begin{rem}
If $\overline{\nabla}$ is a torsion-free affine connection on $M$,
define
\[
\nabla_{X}Y := \overline{\nabla}_X Y  + \frac{1}{2} \left(
\overline{\nabla}_{X}J\right)JY = \frac{1}{2}(\overline{\nabla}_X
Y -J \overline{\nabla}_X JY ),
\]
for $X,Y$ vector fields on $M$. It is easy to see that $\nabla
J=0$ and using \eqref{nijen1}, we obtain that $T(X,Y)=T(JX,JY)$,
i.e. $T$ is of type $(1,1)$ with respect to $J$.

It is proved in \cite[p. 21]{A} that when $\overline{\nabla}$ is
the Levi-Civita connection of a Hermitian metric on $M$, then the
connection thus obtained is the first canonical connection
$\nabla^1$ defined in \eqref{first}.
\end{rem}

\

\section{Complex connections with trivial holonomy}\label{sec-3}

Let $M$ be an $n$-dimensional  connected manifold and $\nabla$ an
affine connection  on $M$ with trivial holonomy. Then the space
$\mathcal{P}^{\nabla}$ of parallel vector fields on $M$ is an
$n$-dimensional real vector space (see for instance
\cite[Proposition 2.2]{W}). If $T$ denotes the torsion tensor
field corresponding to $\nabla$, then
\begin{equation} \label{flat-torsion} T(X,Y)=-[X,Y], \qquad \text{ for all } X,Y \in \mathcal P ^{\nabla} . \end{equation}
We point out that, in general, the space  $\mathcal P^{\nabla}$ of
parallel vector fields is not closed under the Lie bracket. More
precisely, there is the following well-known result (see for
instance \cite[p. 323]{W}):

\begin{lem}\label{p-nabla}
The space $\mathcal P ^{\nabla}$ of parallel vector fields is a
Lie subalgebra of $\mathfrak{X}(M)$ if and only if the torsion $T$
of $\nabla$ is parallel.
\end{lem}

In the next result we give equivalent conditions for an affine
connection with trivial holonomy on an almost complex manifold to
be complex.

\begin{lem}\label{equiv}
Let $M$, $\dim M=2n$, be a connected manifold with an almost
complex structure $J$. Assume that there exists an affine
connection $\nabla $ on $M$ with trivial holonomy. Then the
following conditions are equivalent:
\begin{enumerate}
\item [(i)]  $\nabla J=0$;
\item [(ii)] the space $\mathcal{P}^{\nabla}$ of parallel vector fields is $J$-stable;
\item [(iii)] there exist parallel vector fields  $X_1, \dots , X_n , JX_1, \dots , JX_n$,
linearly independent at every point of $M$.
\end{enumerate}
\end{lem}

\

\subsection{Complex parallelizable manifolds}\label{def(2,0)}
We recall from \cite{Wa} that a complex manifold $(M,J)$ is called
{\em complex parallelizable} when there exist $n$ holomorphic
vector fields $Z_1, \dots , Z_n$, linearly independent at every
point of $M$ \cite{KN,N,Wa}. The following classical result, due
to Wang, characterizes the compact complex parallelizable
manifolds.

\begin{thm}[\cite{Wa}] \label{wang}
Every compact complex parallelizable manifold may be written as a
quotient space $\Gamma\backslash G$ of a complex Lie group by a
discrete subgroup $\Gamma$.
\end{thm}

\medskip

We prove next a result which relates the notion of complex
parallelizability with the existence of a flat complex connection
with torsion of type $(2,0)$.

\begin{prop}\label{complex-parallelizable}
Let $M$ be a connected $2n$-dimensional manifold with a complex
structure $J$. Then the following conditions are equivalent:
\begin{enumerate}
\item [(i)] there exist vector fields $X_1, \dots , X_n , JX_1, \dots , JX_n$, linearly independent at
every point of $M$, such that
\begin{equation}\label{(2,0)}
[X_k,X_l]=-[JX_k,JX_l],\;\;   k<l, \quad [JX_k,X_l]=J[X_k,X_l],
\;\;   k\leq l,
\end{equation}
\item [(ii)] there exist $n$ holomorphic vector fields $Z_1, \dots , Z_n$ which are linearly
independent at every point of $M$ (in other words, $(M,J)$ is
complex parallelizable);
\item [(iii)] there exist $n$ linearly independent holomorphic $(1,0)$-forms $\alpha _1 , \dots , \alpha _n$
on $M$ such that $d\alpha _i $ is a section of  $\Lambda ^{2,0}M$
for every $i$;
\item [(iv)] there exists a complex connection $\nabla$ on $M$ with trivial holonomy whose torsion tensor field $T$ is of type
$(2,0)$.
\end{enumerate}
\end{prop}

\begin{proof}
We recall that a vector field $X$ on $M$ is an infinitesimal
automorphism of $J$ if $[X,JY]=J[X,Y]$ for every vector field $Y$
on $M$.

We note first that (i) is equivalent to (i)', where
\begin{enumerate}
\item[(i)'] there exist vector fields $X_1, \dots , X_n , JX_1, \dots , JX_n$, linearly independent at
every point of $M$, such that $X_j,JX_j$ is an infinitesimal
automorphism of $J$ for each $j$.
\end{enumerate}
The proof of this equivalence is straightforward. Moreover, it
follows from \cite[Proposition IX.2.11]{KN} that (i)' is
equivalent to (ii).

Let $\alpha_1 , \dots , \alpha _n$ be the holomorphic $1$-forms
dual to the holomorphic vector fields $Z_1, \dots , Z_n$. It is
well known that these holomorphic $1$-forms satisfy (iii), and the
converse also holds.

If (iv) holds, then there exist parallel vector fields  $X_1,
\dots , X_n , JX_1, \dots , JX_n$, linearly independent at every
point of $M$ (see Lemma \ref{equiv}). Using \eqref{flat-torsion}
and the fact that $T$ is of type $(2,0)$, relations \eqref{(2,0)}
hold and therefore (i) follows. Conversely, given vector fields
$X_1, \dots , X_n , JX_1, \dots , JX_n$ as in (i), let $\nabla$ be
the affine connection such that the space $\mathcal P^\nabla$ of
parallel vector fields is spanned by these vector fields. It
follows that $\nabla$ has trivial holonomy and Lemma \ref{equiv}
implies that $\nabla J=0$. Moreover, using equations
\eqref{flat-torsion} and \eqref{(2,0)}, we have that $T$ is of
type $(2,0)$ with respect to $J$, and this proves (iv).
\end{proof}

\medskip

\begin{cor}\label{chern flat} Let $(M,J)$ be a complex manifold. The following
conditions are equivalent:
\begin{enumerate}
\item[(i)] $(M,J)$ is complex parallelizable;
\item[(ii)] there exists a Hermitian metric $g$ on $M$ such that the
Chern connection associated to $(M,J,g)$ has trivial holonomy.
\end{enumerate}
\end{cor}

\begin{proof}
(ii) implies (i) follows from Proposition
\ref{complex-parallelizable}.

Assume now that (i) holds. Consider the linearly independent
vector fields $\{X_k, \, JX_k\}$ and the complex connection
$\nabla$ with trivial holonomy given in (i) and (iv) of
Proposition \ref{complex-parallelizable}, respectively. Let $g$ be
the Hermitian metric on $M$ such that the basis above is
orthonormal. Then it follows that $\nabla g=0$, hence by
uniqueness, $\nabla$ is the Chern connection associated to
$(M,J,g)$.
\end{proof}

\smallskip

\begin{rem}
In the compact case, a result similar to Corollary \ref{chern
flat} was obtained in \cite{DLV}. We notice that in \cite{DV,DLV},
a Hermitian metric $g$ on $(M,J)$ whose associated Chern
connection has trivial holonomy is called Chern-flat.
\end{rem}

\medskip

\begin{defn}\label{def-chern-type}
An affine connection $\nabla$ on a connected  complex manifold
$(M,J)$ will be called a {\em Chern-type} connection if it
satisfies condition (iv) of
Proposition~\ref{complex-parallelizable}.
\end{defn}

\medskip

\begin{cor}\label{cor-chern}
Let $(M,J)$ be a connected complex manifold and $\nabla$ an affine
connection with trivial holonomy. Then $\nabla$ is a Chern-type
connection on $(M,J)$ if and only if the space $\mathcal
P^{\nabla}$ of parallel vector fields is $J$-stable and $J$
satisfies
\begin{equation}\label{eq-chern}
J[X , Y]=[X,JY] \quad \text{ for any  } \quad X, Y \in \mathcal
P^{\nabla}.
\end{equation}
\end{cor}

\begin{proof}
We just have to observe that \eqref{eq-chern} is equivalent to
condition (i) of Proposition \ref{complex-parallelizable}.
\end{proof}
\medskip

\begin{rem}
When $J$ is an almost complex structure on $M$, we have the
following equivalences:
\begin{enumerate}
\item[(i)] there exist vector fields $X_1, \dots , X_n , JX_1, \dots , JX_n$, linearly independent at
every point of $M$, such that
\[
[X_k,X_l]=-[JX_k,JX_l],\;\;   k<l, \quad [JX_k,X_l]=[X_k,JX_l],
\;\;   k\leq l,
\]
\item [(ii)] there exists a complex connection $\nabla$ on $M$ with trivial holonomy whose torsion tensor field $T$ is of type
$(2,0)+(0,2)$.
\end{enumerate}
Moreover, analogously to Corollary \ref{cor-chern}, we have that
the torsion $T$ of a complex connection $\nabla$ with trivial
holonomy is of type $(2,0)+(0,2)$ if and only if $J$ satisfies
\begin{equation} [JX,JY]=-[X,Y] \quad \text{ for any } \quad X, Y
\in \mathcal P^{\nabla}.
\end{equation}
\end{rem}

\

\subsection{Flat complex connections with $(1,1$)-torsion}
Given an almost complex structure $J$ on $M$, we study complex
connections $\nabla$ on $M$ with trivial holonomy such that the
corresponding torsion $T$ is of type $(1,1)$ with respect to $J$.
It follows from Corollary~\ref{cor(1,1)} that when the almost
complex structure $J$ admits such a  connection, then $J$
satisfies the integrability condition $N\equiv 0$.


The following proposition is the analogue of Proposition
\ref{complex-parallelizable} in the case when the torsion of the
flat complex connection is of type $(1,1)$.

\begin{prop}\label{def-ab}
Let $M$ be a connected $2n$-dimensional manifold with
 an  almost complex structure $J$. Then the following conditions are equivalent:
\begin{enumerate}
\item [(i)] there exist  vector fields  $X_1, \dots , X_n , JX_1, \dots , JX_n$,   linearly independent at
every point of $M$, such that
\begin{equation} \label{tipo 1-1}
[X_k,X_l]=[JX_k,JX_l], \qquad \;\;[JX_k,X_l]=-[X_k,JX_l], \qquad
k<l;
\end{equation}
\item [(ii)] there exist $n$ commuting vector fields $Z_1, \dots , Z_n$ which are
linearly independent sections of $T^{1,0}M$  at every point of
$M$;
\item [(iii)] there exist $n$ linearly independent $(1,0)$-forms $\alpha _1 , \dots , \alpha _n$
on $M$ such that $d\alpha _i $ is a section of  $\Lambda ^{1,1}M$
for every $i$;
\item [(iv)] there exists a complex connection $\nabla$ on $M$ with trivial holonomy whose torsion tensor field $T$ is of type $(1,1)$.
\end{enumerate}
Moreover, any of the above conditions implies that $J$ is
integrable.
\end{prop}

\begin{proof} Let $X$ and $ Y$ be vector fields on $M$. A simple calculation
shows that
\[ [X-iJX, Y-iJY]=0 \quad {\text{ if and only if }} \quad [X,Y]=[JX,JY] \text{ and }
[JX,Y]=-[X,JY]. \] Thus, if
 (i) holds, $Z_l= X_l-iJX_l$, $l=1, \dots , n$, is a commuting family of
$(1,0)$ vector fields, linearly independent at every point of $M$.
Conversely, given $Z_1 , \dots , Z_n$ as in (ii), setting $X_l=Z_l
+\bar{Z_l}$, it turns out that $X_1, \dots , X_n , JX_1, \dots ,
JX_n$ satisfy (i).

We note first that an almost complex structure satisfying (ii) or
(iii) is integrable (see \cite[Theorem IX.2.8]{KN}). Therefore,
given a $(1,0)$ form $\alpha$, it follows that
\[ d \alpha  \text{ is a section of } \Lambda ^{1,1}M \quad
\text{ if and only if } \quad d \alpha (Z,W)=0 \;\; \forall \, Z,W
\in T^{1,0}M. \] In case $T^{1,0}M$ has a basis $Z_1 , \dots ,
Z_n$  of commuting vector fields, let $\alpha _1, \dots , \alpha
_n$ be the dual basis of $(1,0)$ forms. We calculate
\begin{equation}\label{(1,1)} d \alpha _i(Z_j,Z_l)=Z_j \left(\alpha _i (Z_l)\right) -
Z_l \left(\alpha _i (Z_j)\right)- \alpha_i \left([Z_j
,Z_l]\right),
\end{equation}
where $\alpha _i (Z_k)$ is constant on $M$ and the last summand is
zero since $Z_k$ are commuting vector fields, yielding $d\alpha
_i(Z_j,Z_l)=0$. This clearly implies that $d\alpha _i(Z,W)=0$ for
any $Z,W \in T^{0,1}M$, thus $d\alpha _i $ is a section of
$\Lambda ^{1,1}M$ for every $i$. Conversely, if (iii) holds, let
$Z_1 , \dots , Z_n$ be the basis of $T^{1,0}M$ dual to $\alpha _1,
\dots , \alpha _n$. By assumption, $d\alpha _i(Z_j,Z_l)=0$ for
every $1\leq i, j, l \leq n$, hence \eqref{(1,1)} implies that
$\alpha_i \left([Z_j ,Z_l]\right)=0$ for any $i$, therefore, $[Z_j
,Z_l]=0$ and (ii) follows.

If (iv) holds, then there exist parallel vector fields  $X_1,
\dots , X_n , JX_1, \dots , JX_n$, linearly independent at every
point of $M$ (see Lemma \ref{equiv}). Using \eqref{flat-torsion}
and the fact that $T$ is of type $(1,1)$, relations \eqref{tipo
1-1} hold and therefore (i) follows. Conversely, given vector
fields $X_1, \dots , X_n , JX_1, \dots , JX_n$ as in (i), let
$\nabla$ be the affine connection such that the space $\mathcal
P^\nabla$ of parallel vector fields is spanned by these vector
fields. It follows that $\nabla$ has trivial holonomy and Lemma
\ref{equiv} implies that $\nabla J=0$. Moreover, using equations
\eqref{flat-torsion} and \eqref{tipo 1-1}, we have that $T$ is of
type $(1,1)$ with respect to $J$, and this proves (iv).
\end{proof}

\medskip

The following definition is motivated by Proposition~\ref{def-ab}
(ii).

\begin{defn}
An affine connection $\nabla$ on a connected almost complex
manifold $(M,J)$ will be called an {\em abelian} connection if it
satisfies condition (iv) of Proposition~\ref{def-ab}.
\end{defn}

\medskip

The next corollary is a straightforward  consequence of Lemma
\ref{equiv} and Proposition~\ref{def-ab}.

\smallskip

\begin{cor}\label{cor-abel}
Let $(M,J)$ be a connected  complex manifold and $\nabla$ an
affine connection with trivial holonomy. Then $\nabla$ is an
abelian connection on $(M,J)$ if and only if the space $\mathcal
P^{\nabla}$ of parallel vector fields is $J$-stable and $J$
satisfies
\begin{equation}\label{eq-abel}
[JX , JY]=[X,Y] \quad \text{ for any  } \quad X, Y \in \mathcal P
^{\nabla}.
\end{equation}
\end{cor}

\begin{proof}
We just have to observe that \eqref{eq-abel} is equivalent to
condition (i) of Proposition \ref{def-ab}.
\end{proof}

\

\section{Complete complex connections with parallel torsion and trivial holonomy}

We begin this section by exhibiting a large class of complex
manifolds equipped with complex connections with trivial holonomy
whose torsion tensors are of type $(2,0)$ or $(1,1)$ (compare with
(iv) in Propositions \ref{complex-parallelizable} and
\ref{def-ab}).

\medskip

We begin by recalling known facts on invariant complex structures
and affine connections on Lie groups.

A complex structure on a real Lie algebra $\fg$ is an endomorphism
$J$ of $\fg$ satisfying $J^2=-I$ and such that $N(x,y)=0$ for all
$x,y\in\fg$, where $N$ is defined as in \eqref{nijen}. It is well
known that \eqref{nijen} holds if and only if ${\fg}^{1,0}$, the
$i$-eigenspace of $J$, is a complex subalgebra of $\fg ^{\CC} :=
{\fg}\otimes_\RR \CC$.

When ${\fg}^{1,0}$ is a complex ideal we say that $J$ is
bi-invariant and when ${\fg}^{1,0}$ is abelian we say that $J$ is
abelian. In terms of the bracket on $\fg$, these conditions can be
expressed as follows: \begin{equation} J \text{ is bi-invariant if
and only if} \quad J[x,y]=[x,Jy],\end{equation} and
\begin{equation} J \text{ is abelian if and only if} \quad
[Jx,Jy]=[x,y],\end{equation} for all $x,y \in \fg$. We note that a
complex structure on a Lie algebra cannot be both abelian and
bi-invariant, unless the Lie bracket is trivial.

If $G$ is a Lie group with Lie algebra $\fg$, by left translating
the endomorphism $J$ we obtain a  complex manifold $(G,J)$ such
that left translations are holomorphic maps. A complex structure
of this kind is called left invariant. If $\Gamma \subset G$ is
any discrete subgroup of $G$ with  projection $\pi:G\to\Gamma
\backslash G$ then the induced complex structure on $\Gamma
\backslash G$ makes $\pi$ holomorphic. It will be  denoted $J_0$.

\medskip

Let $G$ be a Lie group with Lie algebra $\fg$ and suppose that $G$
admits a left invariant affine connection $\nabla$, i.e., each
left translation is an affine transformation of $G$. In this case,
if $X,Y$ are two left invariant vector fields on $G$ then
$\nabla_XY$ is also left invariant. Moreover, there is a one-one
correspondence between the set of left invariant connections on
$G$ and the set of $\fg$-valued bilinear forms $\fg \times \fg \to
\fg$ (see \cite[p.102]{He}). It is known that the completeness of
a left invariant affine connection $\nabla$ on $G$ can be studied
by considering the corresponding connection on the Lie algebra
$\fg$. Indeed, the left invariant connection $\nabla$ on $G$ will
be complete if and only if the differential equation on $\fg$ \[
\dot{x}(t) = -\nabla_{x(t)}x(t)\] admits solutions $x(t)\in \fg$
defined for all $t\in \RR$ (see for instance \cite{BM} or
\cite{Gu}).

\smallskip

The left invariant affine connection $\nabla$ on $G$ defined by
$\nabla_XY=0$ for all $X,Y$ left invariant vector fields on $G$ is
known as the $(-)$-connection. This connection satisfies:
\begin{enumerate}
\item Its torsion $T$ is given by $T(X,Y)=-[X,Y]$ for all $X,Y$ left invariant vector fields on $G$;
\item $\nabla T=0$ and $\mathcal P^{\nabla}=\fg \subset
\mathfrak{X}(G)$;
\item The holonomy group of $\nabla$ is trivial, thus, $\nabla$ is flat;
\item The geodesics of $\nabla$ through the identity $e\in G$ are
Lie group homomorphisms $\RR\to G$, therefore, $\nabla$ is
complete;
\item The parallel transport along any curve joining $g\in G$ with
$h\in G$ is given by the derivative of the left translation $(d
L_{hg^{-1}})_g$.
\end{enumerate}

If $\Gamma \subset G$ is any discrete subgroup of $G$ then the
$(-)$-connection on $G$ induces a unique connection on $\Gamma
\backslash G$ such that the parallel vector fields are
$\pi$-related with the left invariant vector fields on $G$, where
$\pi:G\to\Gamma \backslash G$ is the projection. This induced
connection on $\Gamma \backslash G$ is  complete, has trivial
holonomy, its torsion is parallel and $\pi$ is affine.  It will be
denoted $\nabla^0$.

If $J$ is a left invariant complex structure on $G$, then $J$ is
parallel with respect to the $(-)$-connection $\nabla$. Moreover,
$J_0$ is parallel with respect to $\nabla^0$, therefore,
$(\Gamma\backslash G,J_0)$ carries a complete complex connection
with trivial holonomy and parallel torsion.

In the next result we prove that the converse also holds.

\begin{thm} \label{discrete}
The triple $(M,J,\nabla)$ where $M$ is a connected manifold
endowed with a complex structure $J$ and a complex connection
$\nabla$ with trivial holonomy is equivalent to a triple
$(\Gamma\backslash G, J_0, \nabla^0)$ as above if and only if
$\nabla$ is complete and its torsion is parallel.
\end{thm}

\begin{proof} The ``only if" part follows from the previous paragraphs.

For the converse, let  $\mathcal P^{\nabla}$ be the space of
parallel vector fields on $M$.
 According to Lemmas \ref{p-nabla} and \ref{equiv},  $\mathcal
P^{\nabla}$ is a $J$-stable Lie algebra. Let $G$ be the simply
connected Lie group with Lie algebra $\mathcal P^{\nabla}$ with
the left invariant complex structure induced by $J$ on $\mathcal
P^{\nabla}$. This complex structure on $G$ will also be denoted by
$J$. If $\nabla^-$ denotes the $(-)$-connection on $G$, we note
that  $\nabla^-$ is induced by the connection $\nabla$ on $M$.
Moreover, $\nabla ^-$ is a complex connection on $(G, J)$.

Let $e \in G$ be the identity and fix $m\in M$. We identify the
tangent space to $G$ at $e$ with $\mathcal P^{\nabla}$. Let $\phi
: \mathcal P^{\nabla} \to T_m M$ be the complex linear isomorphism
defined by $\phi( X )= X_m $. Since the left invariant vector
fields on $G$ correspond to the parallel vector fields on $M$,
 the Cartan-Ambrose-Hicks Theorem (see \cite[Theorem
1.9.1]{W-libro}) applies and there exists a unique affine covering
$f:(G, \nabla^-) \to (M, \nabla)$ such that $(df)_e = \phi$.
Moreover, since both $\nabla ^-$ and $ \nabla$ are complex
connections and $f$ is affine, it follows that $f$ is holomorphic.

Setting $\Gamma = f^{-1}(m)$, it follows from \cite[Lemmas 2 and
3]{H} that $\Gamma$ acting on $G$ by left translations, coincides
with the deck transformation group of the covering. Therefore, $f$
induces a holomorphic affine diffeomorphism between
$(\Gamma\backslash G, J_0, \nabla^0)$ and $(M,J,\nabla)$.

\end{proof}

Let $G$ be a Lie group with Lie algebra $\fg$ equipped with a left
invariant complex structure $J$ and the $(-)$-connection. Since
the torsion $T$ is given by $T(X,Y)=-[X,Y]$ for all left invariant
vector fields $X,Y$ on $G$, it follows that
\begin{equation}\label{2,0}
T \text{ is of type $(2,0)$ with respect to $J$ if and only if $J$
is bi-invariant on } \fg,
\end{equation}
and
\begin{equation}\label{1-1}
T \text{ is of type $(1,1)$ with respect to $J$ if and only if $J$
is abelian on } \fg.
\end{equation}
These properties are shared also by the induced connection
$\nabla^0$ on $\Gamma\backslash G$ with respect to $J_0$.

\medskip

\begin{cor}\label{coro-chern}
Let $(M,J)$ be a complex manifold with a Chern-type connection
$\nabla$. If the torsion tensor field $T$ is parallel, then:
\begin{enumerate}
\item[(i)] the space $\mathcal P^{\nabla}$ of parallel vector fields on $M$ is
a complex Lie algebra and $J$ is a bi-invariant complex structure
on $\mathcal P^{\nabla}$;
\item[(ii)] if, furthermore, $\nabla$ is complete, then
$(M,J,\nabla)$ is equivalent to $(\Gamma\backslash G,
J_0,\nabla^0)$, where $G$ is a simply connected complex Lie group
and $\Gamma \subset G$ is a discrete subgroup.
\end{enumerate}
\end{cor}

\begin{proof} The space $\mathcal P^{\nabla}$ is a $J$-stable Lie algebra
since $T$ is parallel and $\nabla J=0$. The first assertion now
follows from \eqref{2,0}, noting that a bi-invariant complex
structure on a real Lie algebra gives rise to a complex Lie
algebra.

(ii) follows from (i) and Theorem \ref{discrete}, since $G$ is the
simply connected Lie group with Lie algebra $\mathcal P^{\nabla}$.
\end{proof}

\begin{cor}
Let $(M,J,g)$ a Hermitian manifold such that the associated Chern
connection $\nabla$ is complete, has trivial holonomy and parallel
torsion. Then $(M,J,g)$ is equivalent to a triple
$(\Gamma\backslash G, J_0, g_0)$, where $G$ is a simply connected
complex Lie group and $g_0$ is induced by a left invariant
Hermitian metric on $G$. Furthermore, the Chern connection on the
quotient coincides with $\nabla^0$.
\end{cor}

\begin{proof}
According to Corollary \ref{coro-chern}, we have that
$(M,J,\nabla)$ is equivalent to $(\Gamma\backslash G, J_0,
\nabla^0)$, with $J_0$ a complex structure induced by a
bi-invariant complex structure on $G$.

Following the argument in the proof of Theorem \ref{discrete},
replacing the complex structure by a Hermitian metric, we obtain a
left invariant Hermitian metric on $G$ such that the affine
covering $f:G\to M$ becomes a local isometry. Moreover, the Chern
connection of this Hermitian structure on $G$ is the
$(-)$-connection $\nabla^-$. The induced Hermitian structure on
$\Gamma\backslash G$ is equivalent to the given one on $M$.
\end{proof}

\medskip

Given an affine connection $\nabla$ with torsion $T$, we consider
the tensor field $T^{(2)}$ defined by
\[ T^{(2)}(X,Y,Z,W)=T(T(X,Y),T(Z,W)), \]
where $X,Y,Z,W$ are vector fields on $M$ (see \cite[p. 389]{KT}).

\begin{cor} \label{2-step-solv}
Let $\nabla$ be an abelian connection on a connected complex
manifold $(M,J)$ such that the torsion tensor field $T$ is
parallel. Then:
\begin{enumerate}
\item[(i)] the space $\mathcal P ^{\nabla}$ of parallel
vector fields on $M$ is a Lie algebra and $J$ is an abelian
complex structure on $\mathcal P ^{\nabla}$;
\item[(ii)] the Lie algebra $\mathcal P ^{\nabla}$ is $2$-step solvable, that is,
$T^{(2)}\equiv 0$;
\item[(iii)] if, furthermore, $\nabla$ is complete, then $(M,J,\nabla)$ is equivalent to $(\Gamma\backslash G,J_0,\nabla^0)$, where $G$ is a
simply connected $2$-step solvable Lie group equipped with a left
invariant abelian complex structure and $\Gamma \subset G$ a
discrete subgroup.
\end{enumerate}
\end{cor}

\begin{proof}
It follows from Lemmas \ref{p-nabla} and \ref{equiv} that
$\mathcal P ^{\nabla}$ is a $J$-stable Lie algebra. Corollary
\ref{cor-abel} implies that $J$ is an abelian complex structure on
$\mathcal P ^{\nabla}$. Therefore, $\mathcal P ^{\nabla}$ is
$2$-step solvable (see \cite{P,ABDO}). Moreover, it is
straightforward to verify that the $2$-step solvability of
$\mathcal P ^{\nabla}$ is equivalent to $T^{(2)}\equiv 0$. This
proves (i) and (ii).

(iii) follows from (i), (ii) and Theorem \ref{discrete}, since $G$
is the simply connected Lie group with Lie algebra $\mathcal
P^{\nabla}$.
\end{proof}

\medskip

\begin{rem}
We note that a complete classification of the Lie algebras
admitting abelian complex structures is known up to dimension $6$
(see \cite{ABD}), and there are structure results for arbitrary
dimensions (\cite{BD2}).
\end{rem}

\medskip

\subsection{Examples}

We show next a compact complex manifold $M$ which is not complex
parallelizable but admits abelian connections.

\begin{ex}
Let $N$ be the Heisenberg Lie group, given by
\[  N=\left \{ \begin{pmatrix} 1 & x & z  \\  & 1 & y \\ & & 1 \\
     \end{pmatrix} :\, x,y,z  \in \mathbb R \right\}. \]
The subgroup $\Gamma$ of matrices in $N$ with integer entries is
discrete and cocompact. The $4$-dimensional compact manifold
$M=\left(\Gamma\backslash N\right)\times S^1=(\Gamma\times \mathbb
Z)\backslash (N\times \mathbb R)$ is known as the Kodaira-Thurston
manifold. The Lie group $N\times \mathbb R$ admits a left
invariant abelian complex structure (see for instance \cite{BD,
ABD}), and therefore $M$ inherits a complex structure $J$
admitting an abelian connection. On the other hand, $M$ is not
complex parallelizable. Indeed, if it were, it would follow by
Lemma \ref{wang} that $M$ is a quotient $\Lambda\backslash G$,
where $G$ is a $2$-dimensional complex Lie group and $\Lambda$ is
a discrete cocompact subgroup of $G$. There are only two
$2$-dimensional simply connected complex Lie groups, namely
$\mathbb C^2$ and $\widetilde{\operatorname{Aff}}(\mathbb C)$,
where $\widetilde{\operatorname{Aff}}(\mathbb C)$ is the universal
cover of the group
\begin{equation}\label{aff}
\operatorname{Aff}(\mathbb C)= \left \{
\begin{pmatrix} z & w  \\ 0 & 1 \end{pmatrix} :\, z\in\mathbb
C^{*},\,w\in{\mathbb C} \right\}. \end{equation} The group
$\widetilde{\operatorname{Aff}}(\mathbb C)$ does not admit any
discrete cocompact subgroup, since it is not unimodular
(\cite{Mi}). Thus, $G=\mathbb C^2$ and $M$ is biholomorphic to a
complex torus. Hence, $M$ would admit a K\"ahler structure, which
is impossible since the Kodaira-Thurston manifold does not admit
any such structure (see \cite{Th}).
\end{ex}

\medskip

\begin{ex}
Consider the complex Lie group $\operatorname{Aff}(\mathbb C)$
given in \eqref{aff}, and the discrete subgroup
\[ \Gamma=\left \{
\begin{pmatrix} 1 & w  \\ 0 & 1 \end{pmatrix} :\, w\in{\mathbb Z}[i]
\right\}.\] The quotient $M:=\Gamma\backslash
\operatorname{Aff}(\CC)$ is topologically $\CC^{*}\times
\mathbb{T}^2$. Considering any left invariant Hermitian metric on
$\operatorname{Aff}(\CC)$, we obtain a Hermitian metric on the
quotient whose associated Chern connection is induced by the
$(-)$-connection on the group, hence, it has trivial holonomy.
\end{ex}

\medskip

The next example shows that a complex manifold can admit an
abelian connection with non-parallel torsion.

\begin{ex}
Let $M=\RR^4$ with canonical coordinates $(x_1,x_2,x_3,x_4)$ and
corresponding vector fields $\partial _1, \dots , \partial _4$.
Let $f,g \in C^{\infty}(\RR^4)$, such that $\partial
_k(f)=\partial _k(g)=0$, $k=1,2$, and $\partial _3(f)$ or
$\partial _3(g)$ is not constant. We define an affine connection
$\nabla$  so that the space $\mathcal P^{\nabla}$ of parallel
vector fields is
\[\mathcal P ^{\nabla}=\text{span} _{\, \RR} \{
\partial _1, \partial _2, \partial _3,
 -f \partial _1 +g \partial _2 + \partial _4 \}. \]
We define an almost complex structure on $\RR^4$  as  follows:
\begin{eqnarray*}
J\partial _1 &=&\quad\! \partial _2 , \qquad \quad \hspace{.31cm}
J\partial _3 = -f \partial _1 +g \partial _2+ \partial _4, \\
J\partial _2 &=&-\partial _1 , \qquad \quad\quad \!  J\partial _4
= g \partial _1 +f \partial _2 -\partial _3 .
\end{eqnarray*}
It turns out that $\nabla$ is an abelian connection on $(\RR^4 ,
J)$. In fact, we can check that condition (i) of
Proposition~\ref{def-ab} is satisfied. To do this, we compute
\begin{eqnarray*}
{[J\partial _1, J\partial _3]} &=& {[\partial _2 , -f \partial _1
+g \partial _2+ \partial _4]}= -\partial _2(f) \partial _1
+\partial _2(g) \partial _2=0,
\\
{[\partial _1, J\partial _3]} &=& {[\partial _1 , -f \partial _1
+g \partial _2+ \partial _4]}= -\partial _1(f) \partial _1
+\partial _1(g) \partial _2=0,
\end{eqnarray*}
and the assertion follows. The complex affine manifold
$(\RR^4,J,\nabla)$ cannot be obtained as in Corollary
\ref{2-step-solv}, since
\[ [\partial _3 , J\partial _3]=-\partial _3 (f) \partial _1 + \partial _3 (g)
\partial _2,\]
therefore $\mathcal P^{\nabla}$ is not a Lie algebra, that is, the
torsion tensor field $T$ of $\nabla$ is not parallel, since $f$
and $g$ have been chosen so that $\partial _3(f)$ or
$\partial_3(g)$ is not constant.
\end{ex}

\

\end{document}